\documentclass[
11pt,                          
english                        
]{article}

\usepackage{appendix}
\usepackage{bbm}
\usepackage[english]{babel}    
\usepackage{authblk}
\usepackage{amsmath}           
\usepackage[utf8]{inputenc}    
\usepackage[T1]{fontenc}       
\usepackage{longtable}         
\usepackage{exscale}           
\usepackage[final]{graphicx}   
\usepackage[sort]{cite}        
\usepackage{array}             
\usepackage{eucal} 
\usepackage{hyperref}
\usepackage{wasysym}           
\usepackage[a4paper]{geometry} 
\usepackage[multiuser]{fixme}  
\usepackage{xspace}            
\usepackage{tikz}              
\usepackage{ifdraft}           
\usepackage{chairx}            
\usepackage[expansion=false    
           ]{microtype}        
\usepackage[nottoc]{tocbibind} 
\usepackage[all]{xy}
\usepackage{slashed}
\usepackage{cleveref}
\usetikzlibrary{matrix}
\usepackage{tikz-cd}
\renewcommand{\P}{\mathcal{P}}
\renewcommand{\C}{\mathcal{C}}
\ifdraft{\synctex=1}{}

\fxusetheme{color}

\title{Lie Theory for Complete Curved $A_\infty$-algebras}
\date{}
\author[1]{Niek de Kleijn}
\author[2]{Felix Wierstra\thanks{The second author acknowledges the financial support from Grant  GA CR  No. P201/12/G028.}}
\affil[1]{Korteweg-de Vries Instituut, Univesity of Amsterdam, Science Park 107, Amsterdam, n.dekleijn@uva.nl}
\affil[2]{Faculty of Mathematics and Physics, Charles University, Sokolovsk\'a 49/83, 186 75 Praha 8, Czech Republic, felix.wierstra@gmail.com}

\begin{document}

\maketitle

\begin{abstract}
In this paper we develop the $A_\infty$-analog of the Maurer-Cartan simplicial set associated to an $L_\infty$-algebra and show how we can use this to study the deformation theory of $\infty$-morphisms of algebras over non-symmetric operads. More precisely, we define a functor from the category of (curved) $A_\infty$-algebras to simplicial sets which sends an $A_\infty$-algebra to the associated simplicial set of Maurer-Cartan elements. This functor has the property that it gives a Kan complex. We also show that this functor can be used to study deformation problems over a field of characteristic greater or equal than $0$. As a specific example of such a deformation problem we study the deformation theory of $\infty$-morphisms over non-symmetric operads.
\end{abstract}

\section{Introduction}

Associative algebras up to homotopy, also known as $A_\infty$-algebras, play an important role in many areas of mathematics and mathematical physics. They were originally defined in topology to study loop spaces, but found later applications in representation theory, algebraic geometry, string field theory, mathematical physics, etc. \cite{Stasheff1963,Erler2015, Stasheff1992,Fukaya2003}

The goal of this paper is to define and study the $A_\infty$-analog of the Deligne-Getzler-Hinich groupoid associated to a homotopy Lie algebra, also known as $L_\infty$-algebra. In \cite{Getzler2009}, Getzler associates to every nilpotent $L_\infty$-algebra $L$ a simplicial set $MC_\bullet(L)$. This simplicial set has many good properties and important applications. In \cite{Getzler2009}, it is for example shown that $MC_\bullet(L)$ is a Kan complex.

One important example of an application of the Maurer-Cartan simplicial set of an $L_\infty$-algebra is that the Maurer-Cartan simplicial set can be used in deformation theory to encode the set of deformations  of an object. In this case the zero simplices of $MC_\bullet(L)$, which are the Maurer-Cartan elements of $L$, correspond to the  deformations. Two deformations are equivalent if and only if the corresponding Maurer-Cartan elements are in the same path component of $MC_\bullet (L)$.

It is a well known philosophy that over a field of characteristic $0$, all deformation problems are controlled by the Maurer-Cartan elements in an $L_\infty$-algebra. This philosophy goes back to Deligne, Drinfeld, Feigin, Hinich, Kontsevich-Soibelman, Manetti, and many others, and was made precise by Lurie in \cite{Lurie2010}. Over a field of characteristic $p>0$ this is no longer true, especially since the theory of $L_\infty$-algebras has significant problems over a field of characteristic $p \neq 0$. The theory developed in this paper tries to solve this gap in the case that the deformation problem is controlled by an $A_\infty$-algebra.

The goal of this paper is to define an analog of the Maurer-Cartan simplicial set for $A_\infty$-algebras. Similarly to \cite{Getzler2009}, we define a functor 
\[
MC_\bullet:A_\infty\mbox{-algebras} \longrightarrow  \mbox{Simplicial Sets},
\]
from the category of complete $A_\infty$-algebras to the category of simplicial sets.  We show that for an $A_\infty$-algebra $A$, the corresponding simplicial set $MC_\bullet(A)$ is a Kan complex.

Because the simplicial set $MC_\bullet(A)$ is a Kan complex, we can define an equivalence relation on the set of Maurer-Cartan elements of $A$, where we define two Maurer-Cartan elements to be homotopy equivalent if they are in the same path component of $MC_\bullet(A)$. This allows us to use the Maurer-Cartan simplicial set $MC_\bullet(A)$ to study a deformation theory controlled by $A_\infty$-algebras.

In Section \ref{sec:applications} we give an example of a deformation problem controlled by $A_\infty$-algebras. We explain in this section how the deformation theory of $\infty$-morphisms over non-symmetric operads is controlled by $A_\infty$-algebras. Again the main advantage here is that we no longer have any restrictions on the field we are working over.

\subsection{Comparison with other approaches}

Most of the constructions in this paper are analogous to the $L_\infty$-case. For example our twisting procedure is based on Dolgushev's thesis (see \cite{Dolgushev2005}) and the construction of the Maurer-Cartan simplicial set is based on \cite{Getzler2009}. It seems plausible that many other results from the theory of $L_\infty$-algebras can also be generalized to the $A_\infty$-case (for a few possible generalizations see also the next section about future work).

There are however also a few papers which explicitly deal with versions of some of the constructions we use for $A_\infty$-algebras. The Maurer-Cartan equation for $A_\infty$-algebras has appeared before in for example \cite{Fukaya2003}. There are several papers that explicitly deal with curved $A_\infty$-algebras, see for example the paper by Hamilton and Lazarev \cite{HamiltonLazarev2008} and  the paper by Nicol\'as \cite{Nicolas2008}. Another area in which twisting of $A_\infty$-algebras is defined is in  string field theory (see for example \cite{Erler2015}). The advantage of the work in this paper is that it gives a new point of view to this twisting procedure and that it allows us to work with curved $A_\infty$-algebras. 

In \cite{CHL}, Chuang, Holstein and Lazarev define a notion of a Maurer-Cartan element in an associative algebra and the relation of strong homotopy between these Maurer-Cartan elements. It seems that for associative algebras our definitions coincide.

\subsection{Possible applications, future work and open questions}
 
In this section we sketch some further applications and open questions associated to this paper. 

Given an $A_\infty$-algebra $A$, we can form two different simplicial sets. One by the construction described in this paper and one by taking the Maurer-Cartan simplicial set of the $L_\infty$-algebra corresponding to $A$ with all the higher commutators as $L_\infty$-operations. Although these simplicial sets are clearly not isomorphic, we do expect them to be homotopy equivalent. We have not been able to find a satisfying proof yet and this will be the topic of future work.

A possible further application of the theory developed in this paper is the non-symmetric version of the paper \cite{DHR2015}.  In this paper Dolgushev, Hoffnung and Rogers show that, over a field of characteristic $0$, the category of homotopy algebras over a symmetric operad with $\infty$-morphisms, is enriched over (filtered) $L_\infty$-algebras. By using the Maurer-Cartan simplicial set of an $L_\infty$-algebra this implies that the category of homotopy algebras is also enriched over simplicial sets. It seems highly likely that the theory developed in this paper can be used to answer the question: What do homotopy algebras over a non-symmetric operad form? The main advantage of the theory developed in this paper is that it does not require us to work over a field of characteristic $0$ and would answer this question in a much larger generality than in \cite{DHR2015}. For the sake of briefness we do not work out the details of this construction and leave it as a topic for future work.

The methods developed in this paper do unfortunately not allow us to answer the question: What do homotopy algebras over a symmetric operad over a field of characteristic $p>0$ form? This question still remains open and will be a topic of future research.

Another question that is not treated in full detail in this paper, is the behavior of the Maurer-Cartan simplicial set with respect to quasi-isomorphisms. It seems reasonable to expect that an analog of Theorem 1.1 of \cite{DR2015} also holds for the Maurer-Cartan simplicial set of an $A_\infty$-algebra. This theorem would state that filtered $\infty$-morphisms between filtered $A_\infty$-algebras would induce homotopy equivalences between the corresponding Maurer-Cartan simplicial sets. Since there are several technical differences between our setting and their setting, this generalization is not completely straightforward and we will make this a topic of future work.

\subsection{Conventions}

In this paper we will use the following conventions and notations. We will always work over a field $\K$ of characteristic $p \geq 0$ and always in the category of cochain complexes. We will use a cohomological grading on our cochain complexes, i.e. we use superscripts to indicate the degrees and the differential $d$ will have degree $+1$. We will further assume that all our $A_\infty$-algebras are shifted $A_\infty$-algebras unless stated otherwise. This means that all the products $Q_n$ will have degree $+1$, more details about this are given in Section \ref{sec3}. 

The suspension of a cochain complex $V$ is denoted by $sV$ and is defined by $sV^n=V^{n-1}$. The linear dual of a cochain complex $V$  is denoted by $V^\vee$ and is defined as $\mathrm{Hom}_{\K}(V,\K)$.

All tensor products will be taken over the ground field $\K$, unless stated otherwise. With the notation $V^{\otimes n}$ we will denote $V \otimes ... \otimes V$, where $V$ appears $n$ times and by convention we set $V^{\otimes 0}$ equal to $\K$. We will also implicitly assume that we are using the Koszul sign rule, eg. we assume that the isomorphism $\tau:V\otimes W \rightarrow W \otimes V$ is given by $\tau(v\otimes w)=(-1)^{\vert v \vert \vert w\vert}w \otimes v$, for $v\in V$ and $w \in W$.


\section{Coassociative coalgebras}

In this section we recall the basic definitions for coassociative coalgebras, these coalgebras will play an important role in the definition of $A_\infty$-algebras. This section is mainly meant to fix notation and conventions, we will therefore assume that the reader is familiar with the basic notions of coassociative coalgebras. For more details we refer the reader to Section 1.2 of \cite{LodayVallette}.

Recall that a coassociative coalgebra $C$ is a cochain complex $C$ together with a map $\Delta : C \rightarrow C \otimes C$, which is coassociative and compatible with the differential. We say that a coassociative coalgebra $C$ is counital if there is a map $\epsilon:C \rightarrow \K$ such that this is a counit for the coproduct. A coassociative coalgebra $C$ is called coaugmented if there is an additional map $\eta:\K \rightarrow C$ given, such that $\eta$ is a morphism of coassociative coalgebras. Note that because $\eta$ is a morphism of coassociative coalgebras we get a canonical splitting of $C$ as $C=\K \oplus \mathrm{ker}(\epsilon)$, the ideal $\mathrm{ker}( \epsilon)$ is often denoted by $\bar{C}$ and is called the coaugmentation ideal. The coproduct on $\bar{C}$ will be denoted by $\bar{\Delta}:\bar{C}\rightarrow \bar{C} \otimes \bar{C}$. 

The splitting induced by the coaugmentation defines a pair of adjoint functors between the category of coaugmented coassociative coalgebras and non-counital coassociative coalgebras. The functor from coaugmented coassociative coalgebras is defined by sending  $C$ to its coaugmentation ideal $ \bar{C}$. The adjoint of this functor is defined by sending a non-counital coassociative coalgebra $\bar{C}$ to  the coaugmented coassociative coalgebra $C:=\bar{C} \oplus \K$, where the coaugmentation is defined as the   inclusion of $\K$ into $C$  and the counit as the projection onto $\K$. It is a well known result that these functors define an equivalence of categories.

A coaugmented coassociative coalgebra is called conilpotent if the coradical filtration is exhaustive. The cofree coaugmented conilpotent coassociative coalgebra cogenerated by a cochain complex $V$ is denoted by $T^c(V)$ and is defined as follows. As a cochain complex it is given by $T^c(V):=\bigoplus_{n \geq 0} V^{\otimes n}$, the coproduct is defined by deconcatenation, more explicitly $\Delta$ is given by 
\[\Delta(a_1\cdots a_n)=1\otimes a_1\cdots a_n+a_1\cdots a_n\otimes 1 + \sum_{i=1}^{n-1} a_1\cdots a_i\otimes a_{i+1}\cdots a_n.\]
The coaugmentation is given by the inclusion of $\K$ as $V^{\otimes 0}$ and the  counit is defined as the projection onto $V^{\otimes 0}$. Since the element corresponding to $V^{\otimes 0}$ plays a special role we will denote it by $1$. As is explained in \cite{LodayVallette}, $T^c(V)$ is the cofree coaugmented conilpotent coassociative coalgebra cogenerated by $V$,  thus any morphism $C\rightarrow T^c(V)$ is determined by its image on the cogenerators, i.e. there is a bijection $\mathrm{Hom}_{coalg}(C,T^c(V)) \cong \mathrm{Hom}_{\K}(C,V)$. The cofree conilpotent coaugmented coassociative coalgebra is sometimes also called the tensor coalgebra.

Let $(C,\Delta_C)$ and $(D,\Delta_D)$ be two coassociative coalgebras, then we can equip the tensor product $C \otimes D$ with the structure of a coassociative coalgebra. The coproduct 
\[\Delta_{C \otimes D}:C \otimes D \rightarrow C \otimes D \otimes C \otimes D\]
is given by
\[
C\otimes D \xrightarrow{\Delta_C \otimes \Delta_D} C \otimes C \otimes D \otimes D \xrightarrow{\id \otimes \tau \otimes \id} C \otimes D \otimes C \otimes D,
\]
where $\tau:C \otimes D \rightarrow D \otimes C$ is the flip map. 

As is described in Section 1.3 of \cite{LodayVallette}, the cofree conilpotent coaugmented coassociative coalgebra can be equipped with a natural product called the shuffle product. This product plays an important role in the definition of the twist in Section \ref{sec4} of this paper. The shuffle product is characterized by the following properties, it is a morphism of coassociative coalgebras
\[
\mu_{sh}:T^c(V) \otimes T^c(V) \rightarrow T^c(V),
\]
and on cogenerators it is given by 
\[
\mu_{sh}:T^c(V)\otimes T^c(V) \rightarrow V,
\]
which is defined on $V\otimes \K \oplus \K \otimes V \subset T^c(V) \otimes T^c(V)$ by
\[
\mu_{sh}(1 \otimes v)=\mu_{sh}(v\otimes 1)=v 
\]
and is zero otherwise. Explicitly we have 
\[\mu_{sh}(v_1\cdots v_p\otimes v_{p+1}\cdots v_{p+q})=\sum_{\sigma\in \mbox{\tiny Sh}(p,q)}\epsilon(\sigma)v_{\sigma(1)}\cdots v_{\sigma(p+q)}\] where we denote by 
$\epsilon(\sigma)=\epsilon(\sigma, v_1,\ldots, v_{p+q})$ the Koszul sign and by $\mbox{Sh}(p,q)$ the set of $(p,q)$-shuffles in the symmetric group on $p+q$ letters. With the shuffle product the tensor coalgebra becomes a unital associative algebra, where the unit for the shuffle product is given by the element $1$. It turns out that the shuffle product and the coproduct satisfy the Hopf compatibility relation which is given by 
\[
\Delta \circ \mu_{sh}=(\mu_{sh}\otimes \mu_{sh})\circ (\id \otimes \tau \otimes \id) \circ (\Delta \otimes \Delta ),
\]
where $\tau$ is the flip map. So with the shuffle product and deconcatenation coproduct, the tensor coalgebra  becomes a bialgebra.


A coderivation on a coalgebra $(C,\Delta)$ is a linear map $D\colon C\rightarrow C$ such that 
\[(D\otimes \id +\id\otimes D)\Delta=\Delta D,\] where we remind the reader that we always use the Koszul sign convention. A codifferential on the graded coalgebra $C$ is a degree $+1$ coderivation $Q$ such that $Q^2=0$. Note that we do {\it not} assume that $Q(1)=0$. 
Since $T^c(V)$ is cofree it turns out that any coderivation $D$ on it is uniquely determined by the composite 
\[pr_V\circ D\colon T^c(V)\rightarrow T^c(V)\rightarrow V\] 
where $pr_V$ simply denotes the projection onto the cogenerators.

Denote by \[D_n\colon  V^{\otimes n}\hookrightarrow T^c(V)\stackrel{D}{\longrightarrow}T^c(V)\stackrel{pr_V}\longrightarrow V\] the weight $n$ component of $D$. The coderivation $D$ is determined by the $D_n$  by extending the formula  
\[D(v_1\cdots v_n):=\mu_{sh}(D_0(1), v_1\cdots v_n)+\sum_{p=1}^n\sum_{i=0}^{n-p} v_1\cdots v_iD_p(v_{i+1}\cdots v_{i+p})v_{i+p+1}\cdots v_n\]
to a linear map. This gives a bijection between $\mathrm{Coder}(T^c(V))$ and $\mathrm{Hom}_{\mathbb{K}}(T^c(V),V)$. 






\section{$A_{\infty}$-Algebras}\label{sec3}

In this section we introduce shifted $A_\infty$-algebras, which will be the main objects of study in this paper. 

\begin{definition}
Let $A$ be a graded vector space, a (curved) shifted $A_{\infty}$-algebra structure on $A$ is defined as a codifferential $Q:T^c(A)\rightarrow T^c(A)$ on the cofree coaugmented conilpotent coassociative coalgebra cogenerated by $A$.
\end{definition}

Since $T^c(A)$ is cofree, every derivation is determined by its image on the cogenerators. A shifted $A_\infty$-algebra $A$ is therefore equivalent to a sequence $\{Q_n\}_{n\geq 0}$ of degree $+1$ maps $Q_n\colon A^{\otimes n}\rightarrow A$ satisfying a quadratic condition coming from $Q^2=0$. This gives rise to the ``associativity" conditions
\[\sum_{\substack{a,b,c\geq 0\\ a+b+c=n}}Q_{a+c+1}\circ (\id^{\otimes a}\otimes Q_b\otimes \id^{\otimes c})=0\]
for all $n\geq 0$. 
A flat shifted $A_\infty$-algebra is one for which $Q_0=0$, i.e. $Q(1)=0$. Since in this paper we will be working with curved algebras more often than with flat algebras, we will from now on assume that all $A_\infty$-algebras are curved unless specifically stated to be flat, this is contrary to the usual convention. The associativity conditions above show clearly that a flat shifted $A_\infty$-algebra such that $Q_k=0$ for $k\geq 3$ is simply a shifted dg associative algebra.
Note also that a shifted $A_\infty$-algebra structure on $A$ is equivalent to the usual notion of unshifted $A_\infty$-algebra structure on $sA$, where the multi-products in that setting are given by  $m_n=s\circ Q_n\circ (s^{-1})^{\otimes n}$.

\begin{remark}
An alternative description of flat shifted $A_\infty$-algebras is as algebras over the shifted $A_\infty$-operad. The shifted $A_\infty$-operad is defined as the operadic cobar construction on the coassociative cooperad $As^c$ (concentrated in degree $0$). This equivalence is a straightforward consequence of Theorem 10.1.13 of \cite{LodayVallette}, since they assume that a coderivation maps $1$ to $0$ this is the same as requiring that the shifted $A_\infty$-algebra is flat.
\end{remark}

\begin{definition}[Curvature of an $A_\infty$-Algebra]
The curvature of an $A_\infty$-algebra $(A,Q)$ is the element $Q(1)=Q_0(1)\in A^1$. 
\end{definition}

To facilitate certain infinite sums we will need a topology. To obtain this we consider a decreasing filtration of subspaces
\[A = F^1A\supset F^2A\supset \ldots \] which satisfies $\bigcap_{k}F^kA=\{0\}$. This yields the metric topology given by the metric $d(v,v)=0$ and  $d(v,w)=2^{-|v-w|}$ if $v\neq w$, where $|x|=\max\{ k\mid x\in F^kA\}$. We also assume that the maps $Q_n$ preserve this filtration in the sense that 
\[Q_n(F^{i_1}A\otimes \ldots\otimes F^{i_n}A)\subset F^{i_1+\ldots +i_n}A.\] We will call a shifted $A_\infty$-algebra {\it complete} if this metric is complete. Note that a filtration $F^iA$ on the graded vector space $A$ induces a filtration  $F^iT^c(A)$ in the usual way
\[F^iT^c(A)=\bigoplus_{n\geq 0}\bigoplus_{\substack{i_1,\ldots, i_n\in \N\\ i_1+\ldots+i_n=i}}F^{i_1}A\otimes\ldots\otimes F^{i_n}A.\]  Note also that if $v\in T^c(A)$ then 
\[v=\sum_{n=0}^N\sum_{i=1}^{k_n} \lambda_{i,n}a_{i, 1}\otimes\ldots\otimes a_{i, n}\] 
for some $N, k_n\in \N$, $\lambda_{i,n}\in \K$ and 
$a_{i,j}\in A$. Thus there are numbers $m_{i,j}\in \Z$ such that $a_{i,j}\notin F^{m_{i,j}}A$ and thus, setting $M=N\max\{m_{i,j}\}$, $v\notin F^MT^c(A)$. This implies that 
$\bigcap_{k}F^kT^c(A)=\{0\}$. Again we consider $T^c(A)$ as a metric space for the induced metric as above. This space is not in general complete again and so we denote the completion of $T^c(A)$ by $\widehat{T^c}(A)$. Note that, since the structure maps all respect the filtration, the coalgebra structure and the codifferential extend uniquely to $\widehat{T^c}(A)$ in the appropriate completed sense, eg. the coproduct maps into the completed tensor product see appendix \ref{A}. 

\begin{remark} Note that the ``lower central series" filtration defined by 
\[\mathcal{F}^iA=\sum_{n\geq 1 }\sum_{\substack{i_1, \ldots, i_n\\i_1+\ldots +i_n=i}}Q_n(\mathcal{F}^{i_1}A\otimes \ldots\otimes \mathcal{F}^{i_n}A)\]
for $i>1$ is automatically preserved by the $Q_n$. If $(A, Q)$ is moreover nilpotent, meaning that $\mathcal{F}^iA=0$ for sufficiently large $i$, then the filtration is also complete. In general we may call a shifted $A_\infty$-algebra $(A,Q)$ ``pro-nilpotent" if the lower central series filtration is complete. Note that for a nilpotent shifted $A_\infty$-algebra we find that $\widehat{T^c}(A)= \prod_{n\geq 0}(A)^{\otimes n}$; of course this is true for any filtration that terminates.
\end{remark}

A morphism of $A_\infty$-algebras $F\colon (A,Q_A)\rightarrow (B,Q_B)$, called an $\infty$-morphism, is a degree $0$ counital coalgebra morphism $F\colon T^c(A)\rightarrow T^c(B)$ such that $FQ_A=Q_BF$.
As mentioned above, coalgebra morphisms $C\rightarrow T^c(B)$ are determined by their projection to cogenerators and thus $\infty$-morphisms are completely determined by the linear maps 
\[F_n\colon A^{\otimes n}\rightarrow B\]
for $n\geq 1$ through the formula 
\begin{equation}
    \label{eq:coalgebramorphism}
    \begin{gathered}
      F(a_1\cdots a_n)= 
      \\     
      \sum_{p\geq1}\sum_{\substack{k_1,\ldots, k_p\geq1\\k_1+\ldots+k_p=n}} F_{k_1}(a_1\cdots a_{k_1})\cdots
      F_{k_p}(a_{n-k_p+1}\cdots a_n),
    \end{gathered}
  \end{equation}
  and the assertion that $F(1)=1$. 
  We call a morphism $F\colon (A,Q_A)\rightarrow (B,Q_B)$ a strict morphism if it satisfies $F_n=0$ for all $n\geq 2$. In this case we have $(Q_B)_n\circ (F_1)^{\otimes n}=F_1\circ(Q_A)_n$ for all $n\geq 0$. Note that the set of strict morphisms is in general strictly contained in the set of $\infty$-morphisms.  
 \vspace{0.3cm}
 
 Finally if $(A,Q^A)$ is a flat shifted $A_\infty$-algebra then it yields an underlying cochain complex $(A,Q^A_1)$. Any $\infty$-morphism $F\colon (A,Q^A)\rightarrow (B,Q^B)$ induces a map $ F_1\colon (A,Q^A_1)\rightarrow (B,Q^B_1)$ of the underlying cochain complexes. We will call an $\infty$-morphism of flat shifted $A_\infty$-algebras an $\infty$-quasi-isomorphism if the induced map of cochain complexes is a quasi-isomorphism. 
 
 \begin{remark}
Whenever $\infty$-morphisms occur in this article they will be assumed also to respect the filtrations that are under consideration. Note that any $\infty$-morphism respects the lower central series filtrations. We also impose the stronger condition on $\infty$-quasi-isomorphisms $F\colon A\rightarrow B$ that they induce quasi-isomorphisms $F_kA\rightarrow F_kB$ for all $k\geq 1$. 
 \end{remark}

\subsection{Extension of scalars \label{sec:extensionofscalars}}

Given a unital differential graded associative algebra $(C, \mathbbm{1}, d, \mu)$ and an $A_\infty$-algebra $(A,Q)$ it is possible to equip the tensor product $A \otimes C$ with a new $A_\infty$-structure, which we call the extension of scalars by $C$. Using the identifications $(A\otimes C)^{\otimes n}\cong A^{\otimes n}\otimes C^{\otimes n}$ (using the Koszul sign rule) the maps $CQ_n\colon(A\otimes C)^{\otimes n}\rightarrow A\otimes C$ are given by 
\begin{itemize}
    \item $CQ_0(1)=Q_0(1)\otimes \mathbbm{1}$;
    \item $CQ_1=Q_1\otimes \id + \id \otimes d_C$;
    \item $CQ_k=Q_k\otimes \mu^{(k)}$ for $k\geq 3$.
\end{itemize}
Here $\mu^{(k)}=\mu\circ\mu\otimes \id\circ\ldots\mu\otimes\id^{\otimes k-2}$ denotes the $(k-1)$-fold iterated product of $C$. Note that if $A$ is flat then $A\otimes C$ is automatically flat again. 

\begin{proposition}\label{completeextension}
Let $A$ be an $A_{\infty}$-algebra with filtration $F^iA$ and $C$ a finite dimensional associative algebra, then  the $A_\infty$-algebra $A\otimes C$ equipped with the filtration
\[
F^i(A\otimes C):=(F^iA)\otimes C
\]
is complete if $A$ is complete. 
\end{proposition}
\begin{proof} 

\leavevmode 

\noindent Suppose $(x_n)_{n\in \N}$ is a Cauchy sequence in $A\otimes C$ and let $B\subset C$ be a basis. Then there is a unique decomposition $x_n=\sum_{b\in B}x_{n,b}\otimes b$ for each $n$. Since $(x_n)$ is Cauchy we have that for each $k\in \N$ there is $N\in \N$ such that $\forall n,m>N$ we have $x_n-x_m\in F_kA\otimes C$, i.e.
\[\sum_{b\in B}(x_{n,b}-x_{m,b})\otimes b \in F_kA\otimes C.\] 
Thus we find that the sequences $(x_{n,b})_{n\in \N}$ are Cauchy for each $b\in B$. By completeness of $A$ these last sequences converge to elements $x_b\in A$, i.e. for each $b\in B$ and $k\in \N$ there is $N\in\N$ such that for all $n
>N$ we have $x_b-x_{n,b}\in F_kA$. Thus by finiteness of the set $B$ we find that the sequence $(x_n)$ converges to $x=\sum_{b\in B}x_b\otimes b$. 
\end{proof}


\section{MC elements and twisting}\label{sec4}
As noted above, a shifted $A_\infty$-algebra is defined by a codifferential on the coaugmented cofree conilpotent coassociative coalgebra cogenerated by a $\Z$-graded vector space $A$. This last coalgebra is canonically a bialgebra for the shuffle product.
Note that the shuffle product automatically respects a filtration induced by a filtration on $A$. So if $(A,Q)$ is a complete shifted $A_\infty$-algebra then we also have a bialgebra structure and codifferential on $\widehat{T^c}(A)$, again in the complete sense, see appendix \ref{A}. From now on we will assume completions whenever necessary, but we will not differentiate in notation between maps and their unique extension to completions. 

\begin{definition}[exponential] 
The exponential map 
\[e^{-}\colon A^0\longrightarrow \widehat{T^c}(A)\] is defined as $e^x=\displaystyle\lim_{n\rightarrow \infty} \sum_{k=0}^n x^k$. 
\end{definition}

Note that, letting $\mu_{sh}^{(k)}$ denote the $k-1$th iteration $\mu_{sh}(\mu_{sh}\otimes \id)\ldots(\mu_{sh}\otimes \id^{\otimes k-2})$ for $k\geq 2$, $\mu_{sh}^{(0)}=1$ and $\mu_{sh}^{(1)}=\id$, we could write $e^x=\sum_{k=0}^\infty\frac{1}{k!}\mu_{sh}^{(k)}(x^{\otimes k})$, however the definition above makes perfect sense over a field of arbitrary characteristic.  
\begin{lemma}\label{exp(x)}
The operation $\exp(x)\colon \widehat{T^c}(A)\longrightarrow \widehat{T^c}(A)$ given by 
\[y\mapsto \mu_{sh}(e^x\otimes y)\] for some element $x\in A^0$ defines a coalgebra automorphism with inverse $\exp(-x)$. 
\end{lemma}
\begin{proof}\
\leavevmode 

\noindent 

First of all note that because of the Hopf compatibility relation the multiplication by $e^x$ defines an endomorphism of $\widehat{T^c}(V)$. To show that it is an automorphism first note that $e^x$ satisfies
\[e^x\otimes e^x=\lim_{N\rightarrow \infty} \sum_{n=0}^N\sum_{k=0}^n x^k\otimes x^{n-k}=\lim_{N\rightarrow \infty} \sum_{n=0}^N\Delta(x^n)=\Delta(e^x),\]
i.e. it is a group-like element. 

Similarly we note that 
\[\mu_{sh}(e^x\otimes e^{-x})=\lim_{N\rightarrow \infty}\mu_{sh}(\sum_{k=0}^n\sum_{l=0}^k(-1)^lx^{k-l}\otimes x^l)=1,\]
i.e. $e^x$ is invertible with inverse $e^{-x}$. 

These two facts show that the map $\exp(x)$ defines a coalgebra automorphism of $\widehat{T^c}(A)$ with inverse $\exp(-x)$. 
\end{proof}

\begin{lemma}\label{well-twist}
The operation $Q^x=\exp(-x)\circ Q\circ \exp(x)\colon\widehat{T^c}(A)\longrightarrow \widehat{T^c}(A)$ preserves the subspace $T^c(A)$ 
\end{lemma}
\begin{proof}
\leavevmode

\noindent 

Note that $Q^x$ defines a codifferential by lemma \ref{exp(x)}. Now consider the coderivation $\tilde{Q}^x$ given by the sequence of degree $+1$ maps 
\[\tilde{Q}^x_n\colon A^{\otimes n}\longrightarrow A\]
defined by \[\tilde{Q}^x_n(a_1\cdots a_n)=\sum_{p\geq 0} Q_{n+p}(\mu_{sh}(x^p\otimes a_1\cdots a_n)).\] These are all well-defined by completeness of $A$. Clearly the restriction of $Q^x$ coincides with $\tilde{Q}^x$, which proves the lemma.  
\end{proof}

\begin{definition}[Twisting]
Lemmas \ref{exp(x)} and \ref{well-twist} allow us to twist a shifted $A_\infty$-algebra $(A,Q)$ to the shifted $A_\infty$-algebra $(A^x, Q^x)$, often denoted simply $A^x$, by an element $x\in A^0$. Namely we set $A^x=A$ and 
$Q^x=\exp(-x)\circ Q\circ \exp(x)$ restricted to $T^c(A)$.
\end{definition}

Note that the formula for $\tilde{Q}^x_n(a_1\cdots a_n)$ in the proof of Lemma \ref{well-twist} provides an explicit description of the twisted $A_\infty$-algebra structure.

\begin{definition}[Curvature of an Element]
The curvature of an element $x\in A^0$ in a shifted $A_\infty$-algebra $(A,Q)$ is the element 
\[\mathcal{R}(x):=\exp(-x)Q(e^x)\]
Note, since clearly $\mathcal{R}(x)=Q^x(1)$, it holds that $\mathcal{R}(x)\in A^1\subset T^c(A)$. Thus, since it will not contain any terms of tensor weight higher than $1$, it can be seen that
\[\mathcal{R}(x)=\sum_{l\geq 0}Q_l(x^l)\]
and that $\mathcal{R}(x)$ is simply the curvature of $A^x$.
\end{definition}

\begin{lemma}[Bianchi Identity]\label{Bianchi}
For a shifted $A_\infty$-algebra $(A,Q)$ and $x\in A^0$ we have 
\[Q^x(\mathcal{R}(x))=0\]
or equivalently 
\[\sum_{l\geq 1}Q_l(\mu_{sh}(x^{l-1}\otimes \mathcal{R}(x)))=0.\]
\end{lemma}
\begin{proof} 
\leavevmode

\noindent 

Note that the first identity is obvious since \[Q^x(\mathcal{R}(x))=\exp(-x)Q(\exp(x)\exp(-x)Q(e^x))=\exp(-x)Q^2(e^x)=0. \] Thus it is only left to show that it is equivalent to the second identity. This follows straightforwardly by considering the expression $Q^x(\mathcal{R}(x))=\mu_{sh}(e^{-x}\otimes Q(\mu_{sh}(e^x\otimes \mathcal{R}(x))))$ and realizing that this must be an element of $A$. 
\end{proof}
\begin{definition}[Maurer-Cartan elements] Consider a shifted $A_\infty$-algebra $(A,Q)$. Then $x\in A^0$ is called a Maurer-Cartan element (abbreviated to MC element) if 
\[Q(e^x)=0.\]   
\end{definition}
\begin{corollary}
By invertibility of $e^{x}$ an element $x\in A^0$ is an 
MC element if and only if $\mathcal{R}(x)=0$. So, given a shifted $A_\infty$-algebra $A$ and an element $x\in A^0$, the twisted algebra $A^x$ is flat if and only if $x$ is a Maurer-Cartan element. 
\end{corollary} 

\begin{definition}[Functoriality of Maurer-Cartan elements and twisting]
Given an $\infty$-mor\-phism $F$ of complete shifted $A_\infty$-algebras $(A,Q)$ and $(B,P)$ and an element $x\in A^0$ we  define the element $x_F\in B^0$ called $F$-associated to $x$  as the solution to 
\[e^{x_F}=F(e^x)\] explicitly we have \[x_F:=\sum_{n\geq 1}F_n(x^n).\]
Note that $x_F$ is well-defined by completeness of $B$.
This way we may also define the $\infty$-morphism $F^x\colon A^x\rightarrow B^{x_F}$ as 
\[F^x:=\exp(-x_F)\circ F\circ \exp(x).\] Note finally that if $x$ is a Maurer-Cartan element then so is $x_F$ and, since it can be shown that $x_{F\circ \tilde{F}}=(x_{\tilde{F}})_F$, we obtain a functor associating to an $A_\infty$-algebra its set of Maurer-Cartan elements. 
\end{definition}

\begin{remark}
Note that if $F$ is assumed to be strict in the definition above then $x_F=F(x)$ and $F^x=F$. 
\end{remark}

\section{Cochains on simplices}

In this section we recall the definition of the simplicial standard simplex and  prove some technical results about the normalized cochains on the standard simplex. These technical results will be important in the next sections where we use them to construct the Maurer-Cartan simplicial set and prove that it is a Kan complex.


Recall that the standard $n$-simplex $\Delta^n$ is defined as the simplicial set whose $m$ simplices are given by sequences $(a_0,...,a_m)$, such that $a_i \in \{0,...,n\}$ and $a_i \leq a_{i+1}$ for all $i$. The collection of standard simplices forms a cosimplicial object in the category of simplicial sets. To define the coface and codegeneracy maps we define the following maps 
\[
d^i:\{0,...,n-1\} \rightarrow \{0,..,n\}, \mbox{ for $0\leq i \leq n$,}
\]
\[
d^{i}(k)=
\begin{cases}
k, & \mbox{if } k<i \\
k+1, & \mbox{if } i\leq k.
\end{cases}
\]
and
\[
s^j:\{0,...,n+1\} \rightarrow \{0,...,n\}, \mbox{      for  $0\leq j \leq n$.}
\]
\[
s^{j}(k)=
\begin{cases}
k, & \mbox{if } k<j \\
k-1, & \mbox{if } j\leq k.
\end{cases}
\]
The maps $d^i$ induce the coface maps which we will also denote by $d^i$. On $(a_0,...,a_k)\in \Delta^{n-1}$ the coface map $d^i$ is defined by $d^i(a_0,...,a_k):=(d^i(a_0),...,d^i(a_k))$. The codegeneracy maps $s^j$ are induced by the maps $s^j$ and by abuse of notation we will also denote these maps by $s^j$. If $(a_0,...,a_k)\in \Delta^{n+1}$ then $s^j(a_0,...,a_k):=(s^j(a_0),...,s^j(a_k))$. 

With the maps $d^i$ and $s^j$ the collection of standard simplices forms a cosimplicial object in the category of simplicial sets, we will denote this cosimplicial object by $\Delta^{\bullet}$. For more details see for example \cite{Curtis} or \cite{GoerssJardine}.




The normalized cochains on the simplex $\Delta^n$ with coefficients in the field $\K$, are denoted by $N^\bullet (\Delta^n)$. A basis for the degree $d$ part of the normalized cochains of $\Delta^n$ is given by the functions $\varphi_{i_0,...,i_d}$, for $0\leq i_0<...<i_d\leq n$, where $\varphi_{i_0,...,i_d}$ is the function that evaluates to $1$ on the $d$-dimensional subsimplex of $\Delta^n$ with vertices $i_0,...,i_d$ and is zero otherwise.

Since the normalized cochains are  a contravariant functor, the collection of normalized cochains $\{N^{\bullet}(\Delta^n)\}_{n\geq 0}$ forms a simplicial object in category of  chain complexes. The face maps $\partial_i$ and the degeneracy maps $\sigma_j$  of this simplicial object are defined by precomposing with the maps $d^i$ and $s^j$. More explicitly, let $\varphi_I\in N^{k}(\Delta^n)$ with $I=(i_1,...,i_k)$ then we define the face map $\partial_i$  by the formula $\partial_i(\varphi_I):=\sum_{J \in (d^i)^{-1}(I)}\varphi_{J}$, where the sum ranges over all elements of the inverse image of  $I$ under the map $d^i$, if the inverse image is empty we will define this term to be zero. Because the map $d^i$ is injective the inverse image will consist out of at most one term, for simplicity we will denote this term by $\varphi_{(d^i)^{-1}(I)}$. The degeneracy maps are defined similarly by $\sigma_j(\varphi_I):=\sum_
{J \in (s^j)^{-1}(I)}\varphi_{J}$.


It is a well known fact that the normalized cochains on a simplicial set with the cup product are an associative algebra. In fact they are an  $E_\infty$-algebra, but this  $E_\infty$-structure will be irrelevant for the rest of this paper.  Because this associative algebra structure is functorial, it turns out that the normalized cochains on the simplices also form a simplicial object in the category of associative algebras. We will now recall the formulas for the cup product on the standard simplex. For more details see for example \cite{Hatcher}.

Let $\phi \in N^i(\Delta^n)$ and $ \psi \in N^j(\Delta^n)$, and let $x:\Delta^{i+j} \rightarrow \Delta^n$ be a non-degenerate subsimplex of $\Delta^n$ of dimension $i+j$. The cup product $\cup:N^\bullet(\Delta^n) \otimes N^\bullet(\Delta^n) \rightarrow N^\bullet(\Delta^n) $ of $\phi$ and $\psi$ evaluated on $x$ is then defined by 
\begin{equation}\label{eq:cupproduct}
\phi \cup \psi(x)=\phi(x(0,...,i))\cdot \psi(x(i,...,i+j)),
\end{equation}
where $x(0,...,i)$ is the image of the subsimplex of $\Delta^{i+j}$ with vertices $0,...,i$ under the map $x$. With the cup-product the normalized cochains on the $n$-simplex $\Delta^n$ are turned into a unital associative algebra. The unit  is given by $\mathbbm{1}:=\sum_{i=0}^n \varphi_i$. 






The evaluation on $e_i$, the $i$th vertex of $\Delta^n$, will play a special role in what follows and we will therefore denote it by $\tilde{\epsilon}^i_n:N^\bullet(\Delta^n) \rightarrow \K$. We also define the map $\epsilon^i_n:N^\bullet(\Delta^n) \rightarrow N^\bullet(\Delta^n)$, which is defined as the composition of $\tilde{\epsilon}^i_n$ with the inclusion of the unit $\mathbbm{1}:\K \rightarrow N^\bullet (\Delta^n)$. More explicitly this map is given by $\epsilon^i_n(\varphi_i)=\mathbbm{1}=\sum_{j=1}^n\varphi_j$ and zero otherwise.   To prove that the Maurer-Cartan simplicial set, which we will define in the next section, is a Kan complex we need an explicit contraction on the level of the normalized cochains of $\Delta^n$. To do this we define a contraction between the identity on  $N^\bullet(\Delta^n)$ and $\epsilon^i_n$.

The contraction
\[
h_n^i:N^\bullet(\Delta^n) \longrightarrow N^{\bullet-1}(\Delta^n)
\]
 is defined by 
\[
h^i_n(\varphi_{j_0,...,j_d}):=\sum_{k=0}^d (-1)^i\delta_{i,j_k}\varphi_{j_0,...,\hat{j_k},...,j_d},
\]
where $\hat{j_k}$ means that we omit the element $j_k$ and $\delta_{i,j_k}$ is the Kronecker delta.

\begin{lemma}\label{dh+hd}
The map $h^i_n$ is a contraction between $\id_{N^{\bullet}(\Delta^{n})}$ and $\epsilon^i_n$, i.e. $h^i_n$ satisfies the following formula
\[
d h^i_n+h^i_n d=\id_{N^{\bullet}(\Delta^n)}-\epsilon^i_n.
\]
\end{lemma}

\noindent The proof of the lemma is straightforward but tedious and is left to the reader.

If we have a complete $A_\infty$-algebra $(A,Q_0, Q_1,Q_2,...)$, then we can use the extension of scalars from Section \ref{sec:extensionofscalars} to form a simplicial object in the category of $A_\infty$-algebras. This object is given by $\{A\otimes N^\bullet (\Delta^n) \}_{n\geq 0}$, where the face and degeneracy maps are induced by the face and degeneracy maps of $\{N^\bullet (\Delta^n) \}_{n \geq 0}$. As in section \ref{sec:extensionofscalars} we use the notation $N_\bullet Q_n$  for the $A_\infty$-structure maps on $A \otimes N^\bullet(\Delta^n)$, we will often abuse notation by dropping the $\bullet$.
By using the contraction maps  $h^i_n:N^\bullet (\Delta^n) \rightarrow N^{\bullet-1} (\Delta^n)$ we can also define a ``contraction" of $A \otimes N^\bullet(\Delta^n)$. To do this we first introduce the maps $E^i_n:A \otimes N^\bullet(\Delta^n) \rightarrow A \otimes N^\bullet (\Delta^n)$, which are the analogs of evaluation on the $i$th vertex of $\Delta^n$. This map is  defined as $E_n^i:=\id_A \otimes \epsilon^i_n$. 

When the $A_\infty$-algebra $A$ is flat, then as a map of cochain complexes the map $E_n^i$ is homotopic to the identity and an explicit homotopy $H^i_n:A \otimes N^\bullet (\Delta^n)\rightarrow A \otimes N^{\bullet-1} (\Delta^n)$ is given by setting $H^i_n(a\otimes \varphi):=a \otimes h^i_n(\varphi)$, when $a\otimes \varphi \in A \otimes N^{\bullet}(\Delta^n)$. When $A$ is not flat we can no longer speak about maps of cochain complexes, but the maps $H^i_n$ still have the following property.

\begin{proposition}\label{homotopy} 
The maps $H_n^i$, $E^i_n$ and $\id_{A \otimes N^\bullet(\Delta^n)}$ satisfy the following equation 
\[
H^i_n NQ_1 + NQ_1 H^i_n= \id_{A \otimes N^\bullet(\Delta^n)} -E^i_n.
\]
\end{proposition}

The proof of this proposition follows straightforwardly from Lemma \ref{dh+hd}. For the proof that the Maurer-Cartan simplicial set is a Kan complex we need one more definition.  For each $0 \leq i \leq n $, define the operator $R^i_n:A \otimes N^\bullet(\Delta^n) \rightarrow A \otimes  N^\bullet (\Delta^n)$, as $R^i_n:=NQ_1 \circ H^i_n$. In the following lemmas we give some of the properties of the operators $R^i_n$ and $E^i_n$.

\begin{lemma}\label{EQ=QE}
The following identity holds 
\[E^i_nNQ_l(v)=NQ_l((E^i_n)^{\otimes l}v)\]
for all $n, l\geq 0$, $0\leq i\leq n$ and $v\in \left(A\otimes N^\bullet(\Delta^n)\right)^{\otimes l}$, where for $l=0$ we have used the convention that $(E^i_n)^{\otimes 0}=\id$.
\end{lemma}

\begin{proof}
To prove the lemma we distinguish three different cases, first we prove the lemma for $l=0,1$ and then for $l \geq 2$. When $l=0$ we find that $E^i_nNQ_0=NQ_0$, since $NQ_0(1)=Q_0(1)\otimes \mathbbm{1}$ and $\epsilon^i_n(\mathbbm{1})=\mathbbm{1}$. When $l=1$ the operation $NQ_1$ is given by $Q_1 \otimes \id + \id \otimes d_{N^{\bullet}(\Delta^n)}$. Since $E^i_n$ vanishes on all elements of degree greater than $0$, $E^i_n \circ (\id \otimes d_{N^{\bullet}(\Delta^n)})=0$. Because $d_{N^{\bullet}(\Delta^n)}(\mathbbm{1})=0$, we also have that $(\id \otimes d_{N^{\bullet}(\Delta^n)})E^i_n=0$. We are only left to show that $E^i_n \circ (Q_1 \otimes \id)=(Q_1 \otimes \id) \circ E_n^i$, which is obvious.

The products $NQ_l$ for $l \geq 2$ are given by $Q_l \otimes \cup^{(l)}$. In this case we see that both sides of the equation vanish in case $v$ contains any element of degree higher than $0$ on the second tensor leg as a tensor factor. Furthermore if we work in terms of generators on $A\otimes N^0(\Delta^n)$ induced by the basis $\varphi_j$ for $0\leq j\leq n$ of $N^\bullet(\Delta^n)$ we see that both sides vanish on elements $v$ that are not of the form $(\alpha_1 \otimes \varphi_i)\otimes\ldots\otimes(\alpha_n \otimes \varphi_i)$. For $v=(\alpha_1 \otimes \varphi_i)\otimes\ldots\otimes(\alpha_n \otimes \varphi_i)$ it is clear that $E^i_nNQ_l(v)=NQ_l((E^i_n)^{\otimes l}v)$, which proves the lemma.
\end{proof}

\begin{lemma}\label{Ed=dE}
The maps $E^i_n$ and $R^i_n$ satisfy the following identities:
\begin{enumerate}
\item
\[
\partial_j E_{n}^i=
\begin{cases}
E_{n-1}^i \partial_j, & \mbox{if } i< j \\
E_{n-1}^{i-1} \partial_j, & \mbox{if } i> j.
\end{cases}
\]
\item
\[
\partial_j R_{n}^i=
\begin{cases}
R_{n-1}^i \partial_j, & \mbox{if } i< j \\
R_{n-1}^{i-1} \partial_j, & \mbox{if } i> j.
\end{cases}
\]
\end{enumerate}


\end{lemma}

\begin{proof}
To prove part $1$ of the lemma, first observe that $E_{n-1}^i \partial_j$, $E^{i-1}_{n-1} \partial_j$ and $\partial_j E_{n}^i$ vanish on all elements of the form $\alpha \otimes \varphi$ when $\varphi$ is of degree greater or equal than one. We therefore only need to prove part $1$ of the lemma for elements of the form $\alpha\otimes \varphi_k\in A \otimes N^0(\Delta^{n})$.

Explicit formulas for the face map $\partial_j$ are given by 
\[
\partial_j( \varphi_{i_1,...,i_k})=\varphi_{(d^j)^{-1}(i_1,...,i_k)}=\varphi_{i_1,...,i_{m-1},i_m-1,...,i_k-1},
\]
if none of the $i_l$ is equal to $j$ and where $i_{m-1}$ is the largest $i_{m-1}$ smaller than $j$, if one of the $i_l=j$ then the face map is equal to zero.

After applying these formulas to $\alpha \otimes \varphi_k$, it is straightforward to see that part one of the lemma holds. More explicitly on one side we get we get 
\begin{align*}
\partial_j E^i_n (\alpha \otimes \varphi_k)&=\partial_j (\delta_{i,k}\alpha \otimes \mathbbm{1})\\ &=\delta_{i,k}\alpha \otimes \mathbbm{1}
\end{align*}
and on the other side we get 
\[E^i_n \partial_j (\alpha \otimes \varphi_k)=E^i_n ( \alpha \otimes \varphi_{(d^j)^{-1}(k)})\]
which is non-zero if and only if $k=i$ so this is also equal to $\delta_{i,k} \alpha \otimes \mathbbm{1}$. When $j<i$ the proof is the same except that the last term is non zero if and only if $k$ is $(i-1)$ instead of $i$. This proves part one of the lemma.

To prove the second part of the lemma we need to show that $R^{i}_{n-1}\partial_j=\partial_j R^i_n$ for $i<j$ and $R^{i-1}_{n-1}\partial_j=\partial_j R^i_n$ for $i>j$. We first show it for $i<j$. Let $\varphi_I\in N^k(\Delta^n)$, with $I=(i_1,...i_k)$ and let $\alpha \in A$. In this case we get the following sequence of equalities:
\begin{align*}
\partial_j R^i_{n}(\alpha\otimes \varphi_I)&=\partial_j NQ_1 H^i_n(\alpha \otimes \varphi_{I})\\ &=\partial_j NQ_1(\alpha \otimes \varphi_{I\setminus \{i\}}),
\end{align*}
where we define $\varphi_{I\setminus \{i\}}$ to be zero when the indexing set $I$ does not contain the element $i$. When we continue we get the following terms 
\begin{align*}
\partial_j NQ_1(\alpha \otimes \varphi_{I\setminus \{i\}})&=\partial_j (Q_1(\alpha) \otimes \varphi_{I\setminus \{i\}} \pm \sum_{l=0}^n \pm \varphi_{I \setminus \{i\} \cup l}) \\ &= Q_1(\alpha) \otimes \varphi_{(d^j)^{-1}(I\setminus \{i\})} \pm \sum_{l=0}^n\pm \varphi_{(d^j)^{-1}(I \setminus \{i\} \cup l)}.
\end{align*}
Whenever we take the union of $I$ and $\{l\}$ such that $l$ is already contained in $I$, then we set this term to zero. Because the term $(d^j)^{-1}(I \setminus \{i\} \cup j)$ is zero and $(d^j)^{-1}(l)=l-1$ for $ l>j$, this sum can be rewritten as 
\[
Q_1(\alpha) \otimes \varphi_{(d^j)^{-1}(I\setminus \{i\})} \pm \sum_{l=0}^{n-1}\pm \varphi_{(d^j)^{-1}(I \setminus \{i\}) \cup l}.
\]

When we compute the $R^i_{n-1}\partial_j(\alpha \otimes \varphi_I)$ side for $i<j$ we get the following.

\begin{align*}
R^i_{n-1}\partial_j(\alpha \otimes \varphi_I)&=NQ_1 H^{i}_{n-1} \alpha \otimes \varphi_{(d^j)^{-1}(I)}\\
&=NQ_1(\alpha\otimes \varphi_{(d^j)^{-1}(I)\setminus \{i\}}) \\
&=Q_1(\alpha)\otimes \varphi_{(d^j)^{-1}(I)\setminus \{i\}}\pm \sum_{l=0}^{n-1}\pm \alpha \otimes \varphi_{(d^j)^{-1}(I)\setminus \{i\}\cup \{l\}}.
\end{align*}

Because $i<j$ we have an equality between $(d^j)^{-1}(I)\setminus \{i\}$ and $(d^j)^{-1}(I \setminus \{i\})$. This sum is therefore equal to 
\[
Q_1(\alpha) \otimes \varphi_{(d^j)^{-1}(I\setminus \{i\})} \pm \sum_{l=0}^{n-1}\pm \varphi_{(d^j)^{-1}(I \setminus \{i\}) \cup l},
\]
which is $\partial_j R^i_{n}(\alpha\otimes \varphi_I)$, so for $i<j$ the lemma holds. When $i>j$ a similar arguments show that the lemma also holds in that case. We leave it to the reader to check that the signs agree as well.
\end{proof}


\section{MC set of an $A_\infty$-algebra}
In this section we will define and study the simplicial set of Maurer-Cartan elements associated to a complete shifted $A_\infty$-algebra $(A,Q)$. In the following we will denote the set of Maurer-Cartan elements in a complete shifted $A_\infty$-algebra $(A,Q)$ by $MC(A,Q)$. 
\begin{definition}[Maurer-Cartan simplicial set]
The Maurer-Cartan simplicial set $MC_\bullet(A,Q)$ is given by the 
sets $MC_n(A,Q)=MC(A\otimes N^\bullet(\Delta^n), NQ)$ with the face and degeneracy maps induced by those on $A\otimes N^\bullet(\Delta^n)$. 
Given an $\infty$-morphism $F\colon A\rightarrow B$ we can consider the induced $\infty$-morphisms $N^nF\colon A\otimes N^{\bullet}(\Delta^n)\rightarrow B\otimes N^{\bullet}(\Delta^n)$ given by
$N^n F_l=F_l\otimes \cup^{(l)}$ with the convention that $\cup^{(1)}=\id$. Using these we find that $MC_\bullet$ is a functor from the category of (curved) $A_\infty$-algebras with $\infty$-morphisms to 
simplicial sets. \end{definition}

\begin{remark}
In characteristic $0$ we may associate to any $A_\infty$-algebra an $L_\infty$-algebra by symmetrization. Thus we arrive naturally at the question of comparing the Maurer-Cartan simplicial set constructed in \cite{Hinich, Getzler2009} and the one presented in this paper. As mentioned in the introduction the question of whether these are homotopy equivalent remains open. The main issue is caused by the fact that we use the normalized cochains instead of the polynomial de Rham forms on the standard $n$-simplex. If we would have used the polynomial de Rham forms in the definition above (and stayed in characteristic $0$), then all results in this paper would go through and we would have an isomorphism of the Maurer-Cartan simplicial sets of an $A_\infty$-algebra and the corresponding $L_\infty$-algebra. 

The problem is that one cannot use the polynomial de Rham forms in the characteristic non-zero (as they do not satisfy the Poincar\'e lemma) and one cannot use the normalized cochains in the $L_\infty$ case (as they are not commutative). The comparison of the $L_\infty$ and $A_\infty$ cases in characteristic $0$ comes down to comparing the $A_\infty$-algebra arising by extending scalars by polynomial de Rham forms with the one arising by extending scalars by normalized cochains.
\end{remark}

In \cite{Getzler2009} Getzler shows that the analogous simplicial set for an $L_\infty$-algebra is a Kan complex. We will proceed to show the same. In fact the same exact methods work in this case and so we will be brief in the proofs. To do this recall the maps 
\[R^i_n=NQ_1\circ H^i_n\colon A\otimes N^\bullet(\Delta^n)\longrightarrow A\otimes N^\bullet(\Delta^n)\] and note the following corollary of Proposition \ref{homotopy}.

\begin{corollary}\label{x=Ex+Rx+sum}
Suppose that $A$ is a complete shifted $A_\infty$-algebra, then for all $x\in MC_n(A,Q)$ we have the decomposition 
\[x=E^i_nx+R^i_nx-\sum_{k\geq 2}H^i_nNQ_k(x^k)\]
for all $0\leq i\leq n$.
\end{corollary}
\begin{proof}
This follows form Proposition \ref{homotopy}, the MC equation for $x$ and the fact that $H^i_nNQ_0(1)=0$. 
\end{proof}
In the following we will define $mc^i_n(A,Q)$ as $mc^i_n(A,Q):=\mbox{Im\hspace{0.1cm}}R^i_n$.
\begin{lemma}\label{MC=MCxmc}
Given any (curved) complete shifted $A_\infty$-algebra, the map
\[MC_n(A,Q)\longrightarrow MC(A,Q)\times mc^i_n(A,Q),\]
given by  $x\mapsto (E^i_nx, R^i_nx)$, is a bijection for all $0\leq i\leq n$ and all $n\geq 0$. Here we implicitly equate $MC(A,Q)$ and $MC_0(A,Q)$.  
\end{lemma}
\begin{proof} 
\leavevmode 

\noindent To show surjectivity fix $(e,r)\in MC(A,Q)\times mc^i_n(A,Q)$ and consider the sequence defined recursively by 
\[\alpha_{k+1}=\alpha_0-\sum_{l\geq 2}H^i_nQ_l(\alpha_k^l)\]
where $\alpha_0=e + r$. This sequence is a Cauchy sequence in $A\otimes N^\bullet(\Delta^n)$ and so we may consider its limit $\alpha=\displaystyle\lim_{k\rightarrow \infty}\alpha_k\in A\otimes N^\bullet(\Delta^n)$ by completeness of $A$ and Proposition \ref{completeextension}. By definition of the $\alpha_k$ we have 
\[\alpha=\alpha_0-\sum_{l\geq 2}H^i_nQ_l(\alpha^l),\quad E^i_n\alpha=e\quad\mbox{and}\quad R^i_n\alpha=r.\] This implies that 
\begin{align*} 
NQ_1\alpha &=NQ_1\alpha_0-\sum_{l\geq 2}NQ_1H^i_nNQ_l(\alpha^l)\\
&=\sum_{l\geq 2}H^i_nNQ_1NQ_l(\alpha^l)-\sum_{l\geq 2}NQ_l(\alpha^l)-NQ_0(1).
\end{align*}
Here we used Lemma \ref{EQ=QE}, Proposition \ref{homotopy} and the fact that $e$ is a Maurer-Cartan element. It means we find that 
\[\mathcal{R}(\alpha)=H^i_nNQ_1\mathcal{R}(\alpha)=-\sum_{l\geq1}H^i_nNQ_{l+1}(\alpha^l\otimes\mathcal{R}(\alpha))=0,\]
where the second equality follows from the Bianchi identity (Lemma \ref{Bianchi}) and the third identity follows from the fact that $\bigcap_{k}F^kA=\{0\}$. Note that we have now proved surjectivity of $x\mapsto(E^i_nx,R^i_nx)$. 

\vspace{0.2cm} 

It is left to show injectivity. So suppose $\alpha, \beta\in MC_n(A,Q)$ and $(E^i_n\alpha, R^i_n\alpha)=(E^i_n\beta, R^i_n\beta)$, then by Corollary \ref{x=Ex+Rx+sum} we find that 
\[\alpha-\beta=\sum_{l\geq 2}H^i_nNQ_l(\beta^l-\alpha^l)=
\sum_{l\geq 2}\sum_{k=0}^{l-1}H^i_nNQ_l(\beta^k(\beta-\alpha)\alpha^{k-l-1})=0\]
where the final equality follows again from the fact that $\bigcap_{k}F^kA=\{0\}$. Thus we find that $\alpha=\beta$ which shows injectivity. 
\end{proof}

\begin{proposition}\label{Kanfibration}
Suppose $f\colon (A,Q)\rightarrow (B,P)$ is a surjective strict morphism between complete shifted $A_\infty$-algebras, then the induced map $f\colon MC_\bullet(A,Q)\rightarrow MC_\bullet(B,P)$ is a Kan fibration. 
\end{proposition}
\begin{proof} 
\leavevmode

\noindent For $0\leq i\leq n$ let $\beta\in\boldmath{sSet}(\Lambda^n_i,MC_\bullet(A,Q))$ and let $\gamma$ be an $n$-simplex in $MC_\bullet(B,P)$ such that $\partial_j\gamma=f(\partial_j\beta)$ for $j\neq i$. The map $f\otimes \id\colon A\otimes N^\bullet(\Delta^n)\rightarrow  B\otimes N^\bullet(\Delta^n)$ is a surjective map of simplicial Abelian groups and therefore it is a Kan fibration. Thus there exists an element $\rho\in (A\otimes 
N^\bullet(\Delta^n))^0$ such that $\partial_j\rho=\partial_j\beta$ for all $j\neq i$ and $f\otimes \id(\rho)=\gamma$. Let $\alpha\in MC_\bullet(A,Q)$ be the unique element with $E^i_n\alpha=E^i_n\rho$ and $R^i_n\alpha=R^i_n\rho$ given by Lemma 
\ref{MC=MCxmc}. By Lemma \ref{Ed=dE} we find that $E^i_n\partial_j\alpha=E^i_n\partial_j\beta$ and $R^i_n\partial_j\alpha=R^i_n\partial_j\beta$ for $i\neq j$. Thus, by Lemma \ref{MC=MCxmc}, we find that $\partial_j\alpha=\partial_j\beta$ for $j\neq i$ and $\alpha$ fills the horn $\beta$ in $MC_\bullet(A,Q)$. The facts that $f$ is a strict morphism, 
$f\otimes\id(\rho)=\gamma$, $E^i_n\alpha=E^i_n\rho$ and $R^i_n\alpha=R^i_n\rho$ show that $E^i_nf\otimes \id(\alpha)=E^i_n\gamma$ and $R^i_nf\otimes \id(\alpha)=R^i_n\gamma$. Thus by lemma \ref{MC=MCxmc} we find that $f(\alpha)=\gamma$ and the proposition follows, since $\gamma$ and $\beta$ were arbitrary. 

\end{proof}

\begin{corollary}
Since any complete shifted $A_\infty$-algebra admits a strict map to the trivial shifted $A_\infty$-algebra $0$ we find that $MC_\bullet(A,Q)$ is always a Kan complex. 
\end{corollary}

\begin{proposition}
Suppose $A$ and $B$ are flat shifted $A_\infty$-algebras concentrated in  degrees $-1$ and below, suppose further that $f\colon A\rightarrow B$ is a strict quasi-isomorphism between them, then the induced map on Maurer-Cartan simplicial sets is a homotopy equivalence. 
\end{proposition} 

The proof of this proposition can be taken mutatis mutandis from \cite{Getzler2009}.

\begin{remark} 
Finally we should note that Dolgushev--Rogers expanded the theory developed by Getzler in \cite{DR2015}, in particular they proved a version of the proposition above that does not need the restriction imposed on degree and, moreover, it allows for $\infty$-morphisms instead of only strict morphism. The authors are of the opinion that a similar result can be proved also in the case of $A_\infty$-algebras, however it would be a considerable addition to the present paper to prove it here. Thus we chose to postpone this to future work.

\end{remark}

\section{Application: The deformation theory of $\infty$-morphisms of algebras over non-symmetric operads}\label{sec:applications}

In the last section of this paper, we apply the theory developed in this paper to the deformation theory of $\infty$-morphisms of algebras over non-symmetric operads. The main advantage of our theory is that everything now also works over a field of arbitrary characteristic and not just over a field of characteristic $0$. Most of this section is the non-symmetric version of the results of \cite{Robert-NicoudWierstra1} and \cite{Robert-NicoudWierstra2}. Since all the proofs are completely analogous to the proofs in those papers, we will omit most of them. 

In the remainder of this paper we assume that all operads and cooperads are non-symmetric. We further assume that all operads and cooperads are reduced, i.e. $\P(0)=0$ and $\P(1)=\K$ (resp. $\C(0)=0$ and $\C(1)=\K$). We further assume that all cooperads and coalgebras are conilpotent.

\subsection{$\infty_{\alpha}$-morphisms}

In this section we recall the definition of homotopy morphisms relative to an operadic twisting morphism, we call these morphisms $\infty_\alpha$-morphisms. To do this we need the bar and cobar construction relative to an operadic twisting morphism. We will not recall those here and refer the reader to Chapter 11 of\cite{LodayVallette}. The bar construction relative to a twisting morphism $\alpha:\C \rightarrow \P$ is denoted by $B_\alpha$ and the cobar construction relative to $\alpha$ is denoted by $\Omega_\alpha$. 

\begin{definition}
Let $\alpha:\C \rightarrow \P$ be an operadic twisting morphism from a cooperad $\C$ to an operad $\P$. Let $C$ and $C'$ be $\C$-coalgebras and let $A$ and $A'$ be $\P$-algebras.
\begin{enumerate}
\item An $\infty_\alpha$-morphism, $\Psi:C' \rightsquigarrow C$, from $C'$ to $C$, is defined as a $\P$-algebra map $\Psi:\Omega_\alpha C' \rightarrow \Omega_\alpha C$.
\item An $\infty_\alpha$-morphism, $\Phi:A \rightsquigarrow A'$, from $A$ to $A'$, is defined  as a $\C$-coalgebra map $\Phi:B_\alpha A \rightarrow B_\alpha A'$.
\end{enumerate}
\end{definition}


On the set of $\infty_\alpha$-morphisms from a $\P$-algebra $A$ (resp. $\C$-coalgebra $C'$) to a $\P$-algebra $A'$ (resp $\C$-coalgebra $C$), we can define a notion of homotopy equivalence. This is done by defining a model structure on the categories of $\P$-algebras and $\C$-coalgebras. 

On the category of $\P$-algebras  we define a model structure in which the weak equivalences are given by quasi-isomorphisms, the fibrations by degree-wise surjective maps and the cofibrations are the maps with the left lifting property with respect to acyclic fibrations. A proof that  this is a model structure can be found as Theorem 1.7 in \cite{Harper}. 

On the category of $\C$-coalgebras we define a model structure in which  the weak equivalences are created by the cobar construction, i.e. a map $f:C '\rightarrow C$ is a weak equivalence if the induced map $\Omega_\alpha f: \Omega_\alpha C' \rightarrow \Omega_\alpha C$ is a quasi-isomorphism of $\P$-algebras. The cofibrations are the degree-wise injective maps and fibrations are the maps with the right lifting property with respect to the acyclic cofibrations. This model structure was originally defined by Vallette in \cite{Vallette} for the case that the twisting morphism $\alpha$ is Koszul. This was generalized to general twisting morphisms by Drummond-Cole and Hirsch in \cite{Drummond-ColeHirsch}. 


Since the categories  of $\P$-algebras and $\C$-coalgebras are model categories we have a notion of homotopy between the maps. To make this explicit we need a path object for $\P$-algebras and a cylinder object for $\C$-coalgebras. To define these objects recall that $N_\bullet(\Delta^1)$, the normalized chains on  the $1$-simplex as a coassociative coalgebra, is given by the following coalgebra (see Definition 3.1 of \cite{Vallette}), $N_\bullet(\Delta^1)$ is given by $\K a\oplus \K b \oplus \K c$, with $\vert a \vert =\vert b \vert =0 $ and $\vert c \vert =1$. The coproduct is given by $\Delta (a) =a \otimes a$, $\Delta (b)=b \otimes b$ and $\Delta (c)= a \otimes c +c \otimes b$, the differential is given by $d(a)=d(b)=0$ and $d(c)=b-a$. 

Similar to Section \ref{sec:extensionofscalars}, we can equip the tensor product of a $\C$-coalgebra $C$ and a coassociative coalgebra $A$ with the structure of a $\C$-coalgebra. Denote by $As^c$ the non-symmetric cooperad encoding coassociative coalgebras. The coproduct on $C \otimes A$ is then given by 
\[
\Delta_{C \otimes A}=C\otimes A \xrightarrow{\Delta_C \otimes \Delta_A} (\C \circ C) \otimes (As^c \circ A) \xrightarrow{\cong} (\C \otimes As^c) \circ (C \otimes A) \xrightarrow{\cong} \C \otimes C \otimes A,
\]
where $\circ$ denotes the non-symmetric composition product. In the last line we use that we have a canonical isomorphism between $\C \otimes As^c$ and $\C$. For more details about this isomorphism, see the dual version of  Theorem \ref{thrm:convolutionalgebras}.

\begin{lemma}
\leavevmode
\begin{enumerate}
\item Let $A$ be a $\P$-algebra, then $A \otimes N^\bullet (\Delta^1)$ is a good path object for $A$.
\item Let $C$ be a $\C$-coalgebra, then $C \otimes N_\bullet (\Delta^1)$ is a good cylinder object for $C$. 
\end{enumerate}
\end{lemma}

Using these cylinder and path objects we can define the notion of homotopy between $\infty_{\alpha}$-morphisms.

\begin{definition}
\leavevmode
\begin{enumerate}
\item Let $\Psi,\Psi':C'\rightsquigarrow C$ be two $\infty_\alpha$-morphisms of $\C$-coalgebras, then we call them homotopic if they are homotopic in the model category of $\P$-algebras, in other words if there exists a morphism $H:C' \otimes N_\bullet (\Delta^1) \rightarrow C$, such that the restriction to $C'\otimes a$ is $\Psi$ and the restriction of $H$ to $C'\otimes b$ is $\Psi'$.
\item Let $\Phi,\Phi':A \rightarrow A'$ be two $\infty_\alpha$-morphisms of $\P$-algebras, then we call them homotopic if they are homotopic in the model category of $\C$-coalgebras, i.e. if there exists a map $H:A \rightarrow A'\otimes N^\bullet (\Delta^1)$, such the projection on the first vertex is $\Phi$ and the projection on the second vertex is $\Phi'$.
\end{enumerate}
\end{definition}

\subsection{$A_\infty$-convolution algebras  and the deformation theory of $\infty_\alpha$-morphisms}\label{sec:convolutionalgebras}

Let $\mathcal{C}$ be a cooperad, $\mathcal{P}$ be an operad and $\alpha:\C \rightarrow \P$ be an operadic twisting morphism. Recall from \cite{BergerMoerdijk}, that the convolution operad $\mathrm{Hom}(\C,\P)$ is defined as follows. The arity $n$ component, $\mathrm{Hom}(\C,\P)(n)$, of this operad is defined as $\mathrm{Hom}(\C(n),\P(n))$ the space of linear maps from the arity $n$ component of $\C$ to the arity $n$ component of $\P$. Let $f \in \mathrm{Hom}(\C,\P)(n)$ and $g_1,...,g_n \in \mathrm{Hom}(\C,\P)$, with $g_i \in \mathrm{Hom}(\C,\P)(m_i)$, then we define the composition map by the following sequence of maps
\[
\C \xrightarrow{\Delta_\C} (\C \circ \C) \xrightarrow{\Delta_{\C}} \C(n) \otimes \C(m_1) \otimes ... \otimes \C(m_n) \xrightarrow{f\otimes g_1 \otimes ... \otimes g_n}
\]
\[
\P(n)\otimes \P(m_1) \otimes ... \otimes \P(m_n) \xrightarrow{\gamma_{\P}} \P(m_1+...+m+n),
\]
Where $\Delta_\C$ is the decomposition map of $\C$ and $\gamma_\P$ is the composition map of $\P$. Recall that a twisting morphism $\alpha:\C \rightarrow \P$ is a Maurer-Cartan element in the pre-Lie algebra associated to the convolution operad. The following theorem is the non-symmetric analog of Lemma 4.1 and Theorem 7.1 of \cite{Wierstra1}, see also Section 4 of \cite{Robert-NicoudWierstra1}.

\begin{theorem}
Let $\C$ be a cooperad and let $\P$ be an operad. Then there exists a bijection
\[
\mathrm{Hom}_{Op}(A_{\infty},\mathrm{Hom}(\C,\P)) \cong \mathrm{Tw}(\C,\P),
\]
between the set of operad morphisms from the $A_\infty$-operad to $\mathrm{Hom}(\C,\P)$ and the set of operadic twisting morphisms from $\C$ to $\P$.
\end{theorem}

Using the fact that when we have a $\C$-coalgebra $C$ and a $\P$-algebra $A$, then $\mathrm{Hom}(C,A)$ is a $\mathrm{Hom}(\C,\P)$-algebra (see Proposition 7.1 of \cite{Wierstra1}), we have the following corollary.

\begin{corollary}
Let $C$ be a $\C$-coalgebra and let $A$ be a $\P$-algebra, then $\mathrm{Hom}(\C,\P)$ is a flat $A_\infty$-algebra. The differential $Q_1$ applied to a map $f$ is given by $Q_1(f)=d_A \circ f+ (-1)^{\vert f \vert}f \circ d_C$. The products $Q_n$ for $n \geq 2$ are defined by $Q_n:\mathrm{Hom}(C,A)^{\otimes n} \rightarrow \mathrm{Hom}(C,A)$ is  
\[
Q_n(f_1,...,f_n)(x):=\gamma_A (\alpha\otimes f_1 \otimes ... \otimes f_n)\Delta_C^n(x),
\]
where $\Delta_C^n:C \rightarrow \C(n) \otimes C^{\otimes n}$ is the arity $n$ part of the coproduct of $C$ and $\gamma_A:\P \circ A \rightarrow A$ is the product of $A$. We denote $\mathrm{Hom}(C,A)$ with this $A_\infty$-structure by $\mathrm{Hom}^\alpha(C,A)$.
\end{corollary}

This $A_\infty$-structure has the additional property that the Maurer-Cartan elements in the $A_\infty$-algebra $\mathrm{Hom}^\alpha(C,A)$ correspond to the twisting morphisms relative to $\alpha$. This is the non-symmetric analog of Theorem 2.4 of \cite{Robert-NicoudWierstra2}.

\begin{theorem}\label{thrm:convolutionalgebras}
Let $\alpha:\C \rightarrow \P$ be an operadic twisting morphism and  let $C$ be a $\C$-coalgebra and $A$ a $\P$-algebra. Then the following statements hold.
\begin{enumerate}
\item We have  bijections between the following sets
\[
\mathrm{Hom}_{\C-\mbox{coalgebras}}(C,B_{\alpha}A) \cong MC(\mathrm{Hom}^{\alpha}(C,A)) \cong \mathrm{Hom}_{\P-\mbox{algebras}}(\Omega_\alpha C,A).
\]
\item Two $\C$-coalgebra morphisms $f,g:C \rightarrow B_\alpha A$ are homotopic in the category of $\C$-coalgebras if and only if the corresponding Maurer-Cartan elements are gauge equivalent.
\item Two $\P$-algebra morphisms $f,g:\Omega_\alpha C \rightarrow A$ are homotopic in the category of $\P$-algebras if and only if the corresponding Maurer-Cartan elements are gauge equivalent.
\end{enumerate}
\end{theorem}

Note that, compared to the proof of Theorem 2.4 of \cite{Robert-NicoudWierstra2}, the proof in the non-symmetric case is significantly easier because $N^\bullet(\Delta^1)$ and $N_{\bullet}(\Delta^1)$ are both finite dimensional.

Using Theorem \ref{thrm:convolutionalgebras}, we can now define the deformation complex of $\infty_\alpha$-morphisms for non-symmetric operads. This problem was initially stated by M. Kontsevich in his 2017 S\'eminaire Bourbaki, and was answered for symmetric operads over a field of characteristic $0$ by D. Robert-Nicoud and the second author in Definition 2.7 of \cite{Robert-NicoudWierstra2}. Using the theory developed in this paper we can also answer this question for (co)algebras over a field of arbitrary characteristic.

\begin{definition}
Let $\alpha:\C\rightarrow \P$ be an operadic twisting morphism.
\begin{enumerate}
 \item Let $A$ and $A'$ be two $\P$-algebras, the deformation complex of $\infty_\alpha$-morphisms from $A$ to $A'$ is defined as the $A_\infty$-algebra $\mathrm{Hom}^\alpha(B_\alpha A, A')$.
\item Let $C '$ and $C$ be two $\C$-coalgebras, the deformation complex of $\infty_\alpha$-morphisms from $C'$ to $C$ is defined as the $A_\infty$-algebra $\mathrm{Hom}^\alpha(C ',\Omega_\alpha C)$.
\end{enumerate}
\end{definition}

Because of Theorem \ref{thrm:convolutionalgebras} the Maurer-Cartan elements in this deformation $A_\infty$-algebra correspond indeed to the $\infty_{\alpha}$-morphisms between $A$ and $A'$ (resp $C'$ and $C$), and the notion of gauge equivalence corresponds to the relation of homotopy equivalence. This is therefore the correct deformation complex. For more details see the discussion after Definition 2.7 of \cite{Robert-NicoudWierstra2}. 

As stated earlier the main advantage of working with $A_\infty$-algebras is that we no longer need any restrictions on the ground field we are working over.



\appendix
\section{Completed Coalgebras}\label{A}

In this appendix we will fix what we mean by a coalgebra structure on the completion $\widehat{T^c}(A)$ of the coaugmented conilpotent coassociative cofree coalgebra  cogenerated by the complete filtered $\Z$-graded vector space $A$. In fact it comes from the following general notion. We may consider the category of filtered $\Z$-graded vector spaces $V$ that are complete for the filtration topology. Note that we subsume here the property of the filtration that $\cap_{i\geq 0}F^iV=\{0\}$. Then we may equip this category with a monoidal structure by considering the completed tensor product $\hat{\otimes}$, namely we equip the tensor product of two complete filtered vector spaces with the induced filtration and then complete for the corresponding filtration topology. A coalgebra in the completed sense thus means a coalgebra object in this category. 

Let us be explicit in the case of the tensor coalgebra, since this is the only case that occurs in the present paper. First of all we observe that there exists a canonical isomorphism of (filtered) $\Z$-graded vector spaces
\[\widehat{T^c}(A)\hat{\otimes}\widehat{T^c}(A)\cong \overline{T^c(A)\otimes T^c(A)}\] where on the right hand side the overline indicates completion in the induced filtration topology on $T^c(A)\otimes T^c(A)$. This follows from the fact that the tensor product of two Cauchy sequences is a Cauchy sequence and the inclusion
$T^c(A)\otimes T^c(A)\hookrightarrow\widehat{T^c}(A)\otimes\widehat{T^c}(A)$ of filtered vector spaces. Now the map 
\[\Delta\colon T^c(A)\longrightarrow \overline{T^c(A)\otimes T^c(A)}\] yields the unique extension 
\[\hat{\Delta}\colon \widehat{T^c}(A)\longrightarrow \overline{T^c(A)\otimes T^c(A)}.\]
Similarly to the above we find that 
\[\widehat{T^c}(A)\hat{\otimes}\widehat{T^c}(A)\hat{\otimes}\widehat{T^c}(A)\cong \overline{T^c(A)\otimes T^c(A)\otimes T^c(A)}.\]
Thus it is easily seen that $(\hat{\Delta}\otimes \id)\hat{\Delta}$ and $(\id\otimes \hat{\Delta})\circ\hat{\Delta}$ are both extensions of 
\[\Delta^{(3)}\colon T^c(A)\rightarrow \overline{T^c(A)\otimes T^c(A)\otimes T^c(A)}\] 
where $(\Delta\otimes \id)\circ\Delta=\Delta^{(3)}=(\id\otimes \Delta)\circ\Delta$ denotes the iterated coproduct. So, since such extension is unique, we see indeed that 
$\widehat{T^c}(A)$ forms a coassociative coalgebra in the completed sense.

\vspace{0.3cm} 

The rest of the structures needed in this article follow similarly. Given an $A_\infty$-algebra structure $Q$ we can again extend it uniquely to a map 
\[\hat{Q}\colon \widehat{T^c}(A)\longrightarrow \widehat{T^c}(A)\] 
and, similar to  the coassociativity condition on $\hat{\Delta}$, it can be shown that $\hat{Q}$ defines a coderivation in the completed sense, i.e. where we replace all tensor products in the commuting diagram corresponding to the coderivation property by completed tensor products. Furthermore we find that $\hat{Q}^2=0$. Finally we should consider the shuffle product, which as the comultiplication and coderivations gives rise to the unique extended map 
\[\hat{\mu}_{sh}\colon \overline{T^c(A)\otimes T^c(A)}\longrightarrow \widehat{T^c}(A).\]  Again this map satisfies associativity and together with $\hat{\Delta}$ satisfies the Hopf compatibility condition in the completed sense. 

\vspace{0.3cm} 

In the main body of this article we have chosen to drop the hats on $\hat{\Delta}$, $\hat{Q}$ etc.


\bibliographystyle{plain} 
\bibliography{reference}

\begin{thebibliography}{10}

\bibitem{BergerMoerdijk}
Clemens Berger and Ieke Moerdijk.
\newblock Axiomatic homotopy theory for operads.
\newblock {\em Comment. Math. Helv.}, 78(4):805--831, 2003.

\bibitem{CHL}
Joseph Chuang, Julian Holstein, and Andrey Lazarev.
\newblock Maurer-{C}artan moduli and theorems of {R}iemann-{H}ilbert type.
\newblock 2018.
\newblock arXiv:1802.02549v1.

\bibitem{Curtis}
Edward~B. Curtis.
\newblock Simplicial homotopy theory.
\newblock {\em Advances in Math.}, 6:107--209 (1971), 1971.

\bibitem{Dolgushev2005}
V.~A. Dolgushev.
\newblock {\em A Proof of Tsygan's Formality Conjecture for an Arbitrary Smooth
  Manifold}.
\newblock PhD thesis, MIT, 2005.

\bibitem{DHR2015}
Vasily~A. Dolgushev, Alexander~E. Hoffnung, and Christopher~L. Rogers.
\newblock What do homotopy algebras form?
\newblock {\em Adv. Math.}, 274:562--605, 2015.

\bibitem{DR2015}
Vasily~A. Dolgushev and Christopher~L. Rogers.
\newblock A version of the {G}oldman-{M}illson theorem for filtered
  {$L_\infty$}-algebras.
\newblock {\em J. Algebra}, 430:260--302, 2015.

\bibitem{Drummond-ColeHirsch}
Gabriel~C. Drummond-Cole and Joseph Hirsh.
\newblock Model structures for coalgebras.
\newblock {\em Proc. Amer. Math. Soc.}, 144(4):1467--1481, 2016.

\bibitem{Erler2015}
Theodore Erler.
\newblock Relating berkovits and ${A}_\infty$ superstring field theories; small
  hilbert space perspective.
\newblock {\em Journal of High Energy Physics}, 2015(10):157, Oct 2015.

\bibitem{Fukaya2003}
Kenji Fukaya.
\newblock Deformation theory, homological algebra and mirror symmetry.
\newblock In {\em Geometry and physics of branes ({C}omo, 2001)}, Ser. High
  Energy Phys. Cosmol. Gravit., pages 121--209. IOP, Bristol, 2003.

\bibitem{Getzler2009}
Ezra Getzler.
\newblock Lie theory for nilpotent {$L_\infty$}-algebras.
\newblock {\em Ann. of Math. (2)}, 170(1):271--301, 2009.

\bibitem{GoerssJardine}
Paul~G. Goerss and John~F. Jardine.
\newblock {\em Simplicial homotopy theory}.
\newblock Modern Birkh\"auser Classics. Birkh\"auser Verlag, Basel, 2009.
\newblock Reprint of the 1999 edition [MR1711612].

\bibitem{HamiltonLazarev2008}
Alastair Hamilton and Andrey Lazarev.
\newblock Characteristic classes of {$A_\infty$}-algebras.
\newblock {\em J. Homotopy Relat. Struct.}, 3(1):65--111, 2008.

\bibitem{Harper}
John~E. Harper.
\newblock Homotopy theory of modules over operads and non-{$\Sigma$} operads in
  monoidal model categories.
\newblock {\em J. Pure Appl. Algebra}, 214(8):1407--1434, 2010.

\bibitem{Hatcher}
Allen Hatcher.
\newblock {\em Algebraic topology}.
\newblock Cambridge University Press, Cambridge, 2002.

\bibitem{Hinich}
Vladimir Hinich.
\newblock Descent of {D}eligne groupoids.
\newblock {\em Internat. Math. Res. Notices}, (5):223--239, 1997.

\bibitem{LodayVallette}
Jean-Louis Loday and Bruno Vallette.
\newblock {\em Algebraic operads}, volume 346 of {\em Grundlehren der
  Mathematischen Wissenschaften [Fundamental Principles of Mathematical
  Sciences]}.
\newblock Springer, Heidelberg, 2012.

\bibitem{Lurie2010}
Jacob Lurie.
\newblock Moduli problems for ring spectra.
\newblock In {\em Proceedings of the {I}nternational {C}ongress of
  {M}athematicians. {V}olume {II}}, pages 1099--1125. Hindustan Book Agency,
  New Delhi, 2010.

\bibitem{Nicolas2008}
Pedro Nicol\'{a}s.
\newblock The bar derived category of a curved dg algebra.
\newblock {\em J. Pure Appl. Algebra}, 212(12):2633--2659, 2008.

\bibitem{Robert-NicoudWierstra1}
D.~Robert-Nicoud and F.~Wierstra.
\newblock Homotopy morphisms between convolution homotopy {L}ie algebras.
\newblock 2017.
\newblock arXiv:1712.00794.

\bibitem{Robert-NicoudWierstra2}
D.~Robert-Nicoud and F.~Wierstra.
\newblock Convolution algebras and the deformation theory of
  infinity-morphisms.
\newblock 2018.
\newblock arXiv:1806.03371.

\bibitem{Stasheff1963}
James~Dillon Stasheff.
\newblock Homotopy associativity of {$H$}-spaces. {I}, {II}.
\newblock {\em Trans. Amer. Math. Soc. 108 (1963), 275-292; ibid.},
  108:293--312, 1963.

\bibitem{Stasheff1992}
James~Dillon Stasheff.
\newblock Differential graded {L}ie algebras, quasi-{H}opf algebras and higher
  homotopy algebras.
\newblock In {\em Quantum groups ({L}eningrad, 1990)}, volume 1510 of {\em
  Lecture Notes in Math.}, pages 120--137. Springer, Berlin, 1992.

\bibitem{Vallette}
B.~Vallette.
\newblock Homotopy theory of homotopy algebras.
\newblock 2014.
\newblock arXiv:1411.5533v3.

\bibitem{Wierstra1}
F.~Wierstra.
\newblock Algebraic {H}opf invariants and rational models for mapping spaces.
\newblock 2016.
\newblock arXiv:1612.07762.

\end{thebibliography}

\end{document}